\documentclass[a4paper,12pt]{amsart}
\usepackage[utf8x]{inputenc}
\usepackage[english]{babel}
\usepackage{setspace,amsmath,amssymb,amscd,amsthm,amsfonts}
\usepackage{hyperref}
\usepackage[left=2.5cm,right=2.5cm,top=2.5cm,bottom=2.5cm]{geometry}

\usepackage{import, xifthen, pdfpages, transparent}

\newcommand{\DD}{\Delta\!\!\!\Delta}

\newcommand{\K}{\mathcal K}
\newcommand{\RR}{\mathbb R}
\newcommand{\CC}{\mathbb C}
\newcommand{\ZZ}{\mathbb Z}

\newcommand{\kk}{\mathbf{k}}
\renewcommand{\le}{\leqslant}

\newcommand{\Z}{\mathcal Z}
\newcommand{\R}{\mathcal R}

\newtheorem{formula}{}[section]
\newtheorem{corollary}[formula]{Corollary}
\newtheorem{lemma}[formula]{Lemma}
\newtheorem{theorem}[formula]{Theorem}
\newtheorem{proposition}[formula]{Proposition}
\theoremstyle{definition}

\newtheorem{example}[formula]{Example}

\theoremstyle{remark}
\newtheorem*{remark}{Remark}
\numberwithin{equation}{section}

\DeclareMathOperator{\link}{link}

\DeclareMathOperator{\sk}{sk}

\DeclareMathOperator{\Tor}{Tor}

\begin{document}
\title[Connection between coordinate and diagonal arrangments]{On the connection between coordinate and diagonal arrangement complements}
\author{Vsevolod Tril}
\thanks{The author is supported by a stipend from the Theoretical Physics and Mathematics Advancement Foundation ``BASIS''}
\address{Department of Mathematics and Mechanics, Moscow State University, Russia}
\email{\href{mailto:vsevolod.tril@math.msu.ru}{vsevolod.tril@math.msu.ru}}

%57S12 Toric topology
%14N20 Configurations and arrangements of linear subspaces
%55P10 Homotopy equivalences in algebraic topology
%55P40 Suspensions
%13F55 Commutative rings defined by monomial ideals; Stanley-Reisner face rings; simplicial complexes

\begin{abstract}
	We study diagonal arrangement complements $D(\K)$ in $\CC^m$. We consider the class of simplicial complexes $\K$ in which any two missing faces
	have a common vertex, and prove that the coordinate arrangement complement $U(\K)$ is the double suspension of the diagonal arrangement complement $D(\K)$.
	In the case of subspace arrangements in $\RR^m$ the coordinate arrangement complement $U_{\RR}(\K)$ is the single suspension of $D_{\RR}(\K)$.
\end{abstract}
\maketitle
\section{Introduction}
Subspace arrangements and their complements play an important role in many constructions of combinatorics, 
algebraic and symplectic geometry, and in mechanical systems.
In the work of Arnold \cite{Arn69}, the complement to the arrangement of complex hyperplanes $\{ z_i = z_j \}$ was introduced 
and showed to be the Eilenberg--Mac~Lane space of the pure braid group. 
The cohomology ring of this space was also described in Arnold's work.

The following two classes of arrangements are of particular interest. 
The first class is the coordinate subspace arrangements, which is thoroughly studied in toric topology. 
The complements of coordinate subspace arrangements in $\CC^m$ bijectively correspond to simplicial complexes on the set $[m]$. 
The complement of the coordinate subspace arrangement corresponding to a simplicial complex $\K$ is denoted by $U(\K)$.
In the work \cite{BP00} it was proved that $U(\K)$ is homotopy equivalent to a moment-angle complex $\Z_\K$ 
and using this fact the cohomology ring of coordinate arrangement complements was described. 
The homotopy type of these spaces was explicitly described for some classes of simplicial complexes, 
such as stacked polytopes and shifted complexes, see \cite[\S~4.7 and Chapter~8]{BP}.

Another interesting class is the diagonal arrangements. Their complements were studied in \cite{PRW},\cite{BP00}, \cite{Dobr}. 
In the work \cite{PRW} cohomology groups of real diagonal arrangement
complements were computed via bar-construction of the Stanley--Reisner ring.
Based on these results a connection between cohomology groups of diagonal arrangements and loop spaces on
the polyhedral products was found in \cite{BP00} and developed in \cite{Dobr}.

We study diagonal arrangements in $\CC^m$. The complement of such an arrangement corresponds to a simplicial complex $\K$ on $[m]$ and is denoted by $D(\K)$,
see details in Section~2. We prove the following.
\begin{theorem}[Theorem~\ref{res2}]\label{intro-thrm-1}
	Suppose that any two missing faces of a simplicial complex $\K$ have a common vertex. Then there is a homotopy equivalence
	$$U(\K) \simeq \Sigma^2 D(\K).$$ 
	In case of real arrangement complements there is a homotopy equivalence
	$$U_{\RR}(\K) \simeq \Sigma D_{\RR}(\K).$$
\end{theorem}
From this statement and the ring structure of $H^*(U(\K))$ given in \cite{BP} we obtain
\begin{corollary}\label{intro-cor-1}
	Suppose that any two missing faces of a simplicial complex $\K$ have a common vertex. Then $\K$ is a Golod simplicial complex over any ring $\kk$, that is,
	all products and higher Massey products in $\Tor^*_{\kk[v_1, \dots, v_m]}(\kk[\K], \kk)$ are trivial.
\end{corollary}
Using the results of \cite{IK} we obtain a decomposition of the double suspension of $D(\K)$ for $\K$ satisfying the condition of Theorem~\ref{intro-thrm-1}:
\begin{corollary}\label{intro-cor-2}
	Suppose that any two missing faces of a simplicial complex $\K$ have a common vertex.
    Then there is a homotopy equivalence
	$$\Sigma^2 D(\K) \simeq \bigvee_{\varnothing \neq I \subset [m]} \Sigma^{|I| + 1} |\K_I|.$$
\end{corollary}

The work is organized as follows. Section~2 contains the preliminary definitions and constructions.
In Section~3 we apply the results of \cite{dLS} to calculate the cohomology ring for a class of diagonal arrangement complements. 
In Section~4 we prove that every coordinate arrangement complement is homotopy equivalent to some diagonal arrangement complement and prove the main
theorem. In Section~5 we combine our results with the results of \cite{IK} to obtain Corollary~\ref{intro-cor-1} and Corollary~\ref{intro-cor-2}.
We also give examples of simplicial complexes which satisfy the condition of Theorem~\ref{intro-thrm-1} but are not contained in the classes of simplicial complexes considered in \cite{IK}.

\subsection*{Acknowledgements} I wish to express gratitude to my advisor Taras~Panov for stating the problem, help, support and valuable advice.

\section{Basic definitions}
\textit{An abstract simplicial complex} on the set $V$ is a collection $\K$ of subsets $I \subset V$, called \textit{simplices},
that satisfies the following conditions:
\begin{enumerate}
	\item[$\bullet$] $\varnothing \in \K$,
	\item[$\bullet$] if $I \in \K$ and $J \subset I$ then $J \in \K$,
\end{enumerate}
One-element simplices $\{ i \} \in \K$ are called \textit{vertices}. One-element subsets $\{i \} \in V$ such that $\{i \} \notin \K$ are called \textit{ghost vertices}.

A \textit{missing face} of simplicial complex $\K$ is a subset $I \subset V$ such that $I \notin \K$, but every
proper subset of $I$ is a simplex of $\K$.
The set of all missing faces of $\K$ is denoted by $MF(\K)$.

Usually, $V$ is the set $[m] = \{1, 2, \dots, m \}$.

{\it An arrangement} is a finite set $\mathcal{A} = \{L_1, \dots, L_r \}$ of affine subspaces in some affine space (either real or complex).

Given an arrangement $\mathcal{A} = \{L_1, \dots, L_r\}$ in $\CC^m$, define its union $|\mathcal{A}|$ as
$$|\mathcal{A}| = \bigcup_{i = 1}^r L_i \subset \CC^m,$$
and its complement $M(\mathcal{A})$ as
$$M(\mathcal{A}) = \CC^m \backslash |\mathcal{A}|,$$
and similarly for arrangements in $\RR^m$.

Suppose that $\mathcal{A} = \{L_1, \dots, L_r \}$ is an arrangement in $\CC^m$. 
The intersections
	$$v = L_{i_1} \cap \dots \cap L_{i_k}$$
form a poset $(P, <)$ with respect to the inverse inclusion:
$u < v$ if and only if $v$ is a proper subspace in $u$.
The minimal element $\perp$ of this poset is $\CC^m$, 
and the maximal element $\top$ is $\bigcap_{i = 1}^r L_i$.

This poset is called the {\it intersection poset} of arrangement. 
The intersection poset is a lattice: the join and meet operations are defined by
$$u \vee v = u \cap v, \text{	} u \wedge v = \bigcap_{L: \text{ } u \in L,\text{ } v \in L} L.$$
The \textit{rank function} $d$ on $P$ is defined by
$d(u) = \dim(u)$.

An arrangement 
$\mathcal{A} = \{L_1, \dots, L_r \}$ is called \textit{coordinate}
if every subspace
$L_i$, $i = 1, \dots, r$, is a coordinate subspace.

Every coordinate subspace in $\CC^m$ can be written as
	$$C_I = \{(z_1, \dots, z_m) \in \CC^m : z_{i_1} = \dots = z_{i_k} = 0 \}, $$
where $I = \{i_1, \dots, i_k \}$ is a subset in $[m]$.

For each simplicial complex $\K$ on the set $[m]$ define the \textit{complex coordinate arrangement}
$\mathcal{CA}(\K)$ by
	$$\mathcal{CA}(\K) = \{C_I : I \notin \K \}.$$
Denote the complement of this arrangement by $U(\K)$, that is
	$$U(\K) = \CC^m \backslash \bigcup_{I \notin \K} C_I.$$
The real coordinate arrangement and its complement $U_\RR(\K)$ is defined in the same way.
\begin{proposition}[\cite{BP00}, Proposition 5.2.2]
	The assignment
	$\K \mapsto U(\K)$ defines a one-to-one order preserving correspondence between the set of simplicial complexes on $[m]$
	and the set of coordinate arrangement complements in $\CC^m$ (or $\RR^m$).
\end{proposition}
Coordinate subspace arrangement complements are the examples of polyhedral product spaces.

Let $\K$ be a simplicial complex on $[m]$ and let
$$(\textit{\textbf X}, \textit{\textbf A}) = \{ (X_i, A_i), i = 1, \dots, m \}$$
be a collection of $m$ pairs of spaces.
The \textit{polyhedral product} of $(\textit{\textbf X}, \textit{\textbf A})$ corresponding to $\K$ 
is the topological space in $X_1 \times \dots \times X_m$ given by
$$(\textit{\textbf X}, \textit{\textbf A})^\K = \bigcup_{I \in \K} Y_1 \times \dots \times Y_m,
\text{ where } Y_i = \begin{cases} X_i, &\text{if } i \in I,\\ A_i, &\text{if } i \notin I. \end{cases}$$

We can see that $U(\K) = (\CC, \CC^\times)^\K$. 
Other examples of polyhedral products are the \textit{moment-angle complex}
$$\Z_\K = (D^2, S^1)^\K,$$
and the \textit{real moment-angle complex}
$$\R_\K = (D^1, S^0)^\K.$$
\begin{theorem}[\cite{BP}, Theorem~4.7.5]
	The moment-angle complex $\Z_\K$ is a $T^m$-invariant subspace in $U(\K)$, and there is a $T^m$-equivariant 
	deformation retraction
	$$\Z_\K \hookrightarrow U(\K) \stackrel{\simeq}{\to} \Z_\K.$$
	The same is true in the real case.
\end{theorem}

For each subset $I = \{ i_1, \dots, i_k\} \subset [m]$ define the \textit{diagonal subspace} $D_I$ in $\CC^m$ as
	$$D_I = \{(z_1, \dots, z_m) \in \CC^m : z_{i_1} = \dots = z_{i_k} \}. $$
A configuration $\mathcal{A} = \{L_1, \dots, L_r \}$ is called \textit{diagonal} if every subspace $L_i$, $i = 1, \dots, r$, is a diagonal subspace.
Given a simplicial complex $\K$ on the vertex set $[m]$ (without ghost vertices), the \textit{diagonal subspace arrangment} $\mathcal{DA}(\K)$
is the set of subspaces $D_I$ such that $I$ is not a simplex of $\K$:
$$\mathcal{DA}(\K) = \{D_I : I \notin \K \}.$$
Denote the complement of the arrangement $\mathcal{DA}(\K)$ by $D(\K)$.

\begin{proposition}[\cite{BP00}, Proposition~5.3.2]
The assignment $\K \mapsto D(\K)$ defines a one-to-one order preserving correspondence between the set of simplicial complexes on
the vertex set $[m]$ and the set of diagonal subspace arrangement complements in $\CC^m$ (or $\RR^m$).
\end{proposition}

\section{The cohomology ring of the complements of arrangements}
The cohomology groups of a diagonal arrangement complements can be computed via its intersection lattice using the Goresky--MacPherson formula \cite{GM}.
Given a poset $Q$, denote its order complex by $\Delta(Q)$.
For elements $a, b$ in the poset $Q$ we denote by $[a, b]$, $(a, b]$ and $[a, b)$ 
the respective intervals in $Q$. The corresponding order complexes are denoted by $\Delta[a, b]$, $\Delta(a, b]$, $\Delta[a, b)$.

\begin{theorem}[\cite{GM}, Part~III; \cite{dLS}, Proposition~3.1]\label{GorMac}
	Let $\mathcal{A}$ be an arrangement in $\RR^N$ with intersection lattice $P$.
	Denote by $\DD[\perp, u]$ the pair of spaces 
	$$\DD[\perp, u] = (\Delta[\perp, u], \Delta(\perp, u] \cup \Delta[\perp, u)).$$

	Then the cohomology groups of the complement $ $ are given by
	$$H^q(M(\mathcal{A})) \cong \bigoplus_{u \in P} H_{N - d(u) - q}(\DD[\perp, u]).$$
\end{theorem}
For a particular family of diagonal arrangements the cohomology groups can be calculated from the links or full subcomplexes in $\K$.
Recall that the \textit{link} of a simplex $I \in \K$ is the subcomplex 
$$\link_{\K} I = \{J \in \K : I \cup J \in \K, I \cap J = \varnothing \},$$
and for $J \subset [m]$ the \textit{full subcomplex} of $\K$ on the set $J$ is 
$$\K_J = \{I \in \K : I \subset J \}.$$
\begin{proposition}\label{GorMac2}
	Suppose that any two missing faces of $\K$ have a common vertex. Then there is an additive isomorphism
	\begin{equation}\label{Alink}
		H^q(D(\K)) \cong \bigoplus_{\widehat I \in \widehat \K} H_{2m - 2|\widehat I| - q - 4}(\link_{\widehat \K}\widehat I),
	\end{equation}
	\begin{equation}\label{subcompl}
		H^q(D(\K)) \cong \bigoplus_{I \subset [m]} \widetilde H^{q - |I| + 1}(\K_I).
	\end{equation}
\end{proposition}
\begin{proof}
Note that if $J_1 \cap J_2 \neq \varnothing$, then the intersection of the diagonal subspaces $D_{J_1}$ and $D_{J_2}$ is the diagonal subspace $D_{J_1 \cup J_2}$.
It follows that if any two missing faces of $\K$ have a common vertex, then every element of the intersection lattice of the arrangement $\mathcal{DA}(\K)$ 
is $D_I = \{z_{i_1} = \dots = z_{i_k}\}$ for some $I = \{i_1, \dots, i_k \} \notin \K$.
Therefore, the intersection lattice is isomorfic to the poset of non-faces of $\K$ ordered by inclusion.

In other words, the intersection lattice of $\mathcal{DA}(\K)$ is isomorphic to the set of faces of the Alexander dual complex $\widehat \K$
ordered by reverse inclusion with maximal element $\top = \varnothing$ and added minimal element $\perp$.

In \cite{dLS} the following isomorphism was proved:
\begin{equation*}
	H_k(\DD[\perp, u]) \cong \begin{cases}
			\widetilde{H}_{k-2}(\Delta(\perp, u)) & u > \perp,\\
			H_k(pt) & u = \perp,
	\end{cases}
\end{equation*}
where we assume that $\widetilde H_{-1}(\varnothing) \cong \ZZ$.

Since the intersection lattice of $\mathcal{DA}(\K)$ is isomorphic to $\widehat \K$, the order complex
$\Delta(\perp, u)$ for $u = \{z_{i_1} = \dots = z_{i_k}\}$ is isomorphic to the barycentric subdivision of the simplicial complex
$\link_{\widehat \K} \widehat I$, where $\widehat I = [m] \backslash I = \{i_1, \dots, i_k \}$. 
Thus, we have
$$H_k(\DD[\perp, u]) \cong H_{k-2}(\link_{\widehat \K} \widehat I)$$
and from Theorem~\ref{GorMac} we obtain the formula (\ref{Alink}).

To prove (\ref{subcompl}) we use the combinatorial Alexander duality \cite[Corollary~2.4.6]{BP}.
We have
$$\widetilde H_{2m-2|\widehat I| - q - 4}(\link_{\widehat \K} \widehat I) = \widetilde H_{2|I| - q - 4} (\link_{\widehat \K} \widehat I)
	\cong \widetilde H^{q - |I| + 1} (\K_I). \qedhere $$
\end{proof}
\begin{remark}A similar formula for coordinate arrangement complements was obtained in \cite{BP}.
The connection between these two results will be described below in Theorem~\ref{res2}.
\end{remark}
The ring structure in cohomology of the so-called $(\ge 2)$-arrangement complements was described in \cite{dLS}. 
We need only the case of complex arrangements in $\CC^m$,
which are always $(\ge 2)$-arrangements.
The work \cite{dLS} uses the following combinatorial and algebraic constructions to describe the multiplication in the Goresky--MacPherson formula.

In the notation of Theorem~\ref{GorMac}, we have a map
$\times : H_k(\DD[\perp, u]) \otimes H_l(\DD[\perp, v]) \to H_{k+l}(\DD([\perp, u] \times [\perp, v]) )$ 
defined on the simplicial chains by
$$\times : C_k(\Delta(P)) \otimes C_l(\Delta(Q)) \to C_{k+l}(\Delta(P \times Q)),$$
$$\langle u_0, \dots, u_k \rangle \otimes \langle v_0, \dots, v_l \rangle
\mapsto \sum_{ \substack{
	0 = i_0 \le \dots \le i_{k+l} = k+l \\
	0 = j_0 \le \dots \le j_{k+l} = k+l \\
	(i_r, j_r) \neq (i_{r+1}, j_{r+1}) } }
\sigma_{i, j} \langle (u_{i_0}, v_{j_0}), \dots, (u_{i_{k+l}}, v_{j_{k+l}}) \rangle.
$$
The signs $\sigma_{i, j}$ are defined so that $\sigma_{i, j} = 1$ if $k = 0$ or $l = 0$ 
and the Leibniz rule $\partial(s \times t) = \partial s \times t + (-1)^k s \times \partial t$ is satisfied.

Second, the join operation
$$\vee: [\perp, u] \times [\perp, v] \to [\perp, u \cap v].$$
$$(z, w) \mapsto z \cap w $$
preserves the ordering and therefore induces a simplicial map of the associated order complexes.
If subspaces $u$ and $v$ satisfy the so-called \textit{codimension condition}
$$d(u) + d(v) - d(u \cap v) = 2m \quad \Leftrightarrow \quad u + v = \CC^m$$
then $\vee$ is an injective map and induces a simplicial map of pairs
$$ \vee_*: \DD[\perp, u] \times \DD[\perp, v] \to \DD[\perp, u \cap v].$$
Note that for any two lattices $P$ and $Q$ we have a homeomorphism $\Delta(P) \times \Delta(Q) \cong \Delta(P \times Q)$.

\begin{theorem}[\cite{dLS}, Theorem 5.2]\label{product}
	Let $\mathcal{A}$ be a complex arrangement. Then the product of cohomology classes in Theorem~\ref{GorMac} is given by:
	$$H_k(\DD[\perp, u]) \otimes H_l(\DD[\perp, v]) \to H_{k+l}(\DD[\perp, u \cap v]),$$
	\begin{equation*}
		a \otimes b \mapsto
		\begin{cases} 
			%(-1)^{l(2m - d(u))} 
   \vee_*(a \times b), &\text{if } d(u) + d(v) - d(u \cap v) = 2m,\\
			0, &\text{otherwise}.
		\end{cases}
	\end{equation*}
\end{theorem}
\begin{example}\label{D-RP2}
	Let $\K$ be the 6-vertex triangulation of projective plane shown in the Figure~\ref{fig:RP2}.
	\begin{figure}[ht]
		\centering
	\def\svgscale{1}
	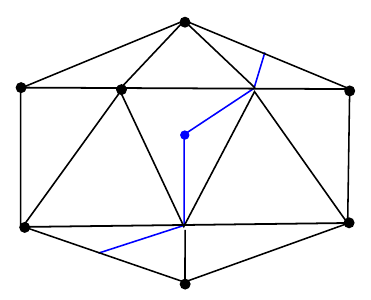

		\caption{The triangulation $\RR P^2$}
		\label{fig:RP2}
	\end{figure}
	We can find non-zero cohomology groups of $D(\K)$ using Proposition~\ref{GorMac2}:
	$$
		H^0(D(\K)) = \ZZ, \quad H^3(D(\K)) = \ZZ^{10},\quad H^4(D(\K)) = \ZZ^{15},$$
		$$H^5(D(\K)) = \ZZ^{6},\quad H^7(D(\K)) = \ZZ_2.
	$$
	The only possible non-trivial product is 
	$$H^3(D(\K)) \otimes H^4(D(\K)) \to H^7(D(\K)).$$
	By Theorem~\ref{GorMac} we have
	$$H^3(D(\K)) \cong \bigoplus_{I \notin \K, |I| = 3} H_1(\DD[\perp, D_I]), \quad H^4(D(\K)) \cong \bigoplus_{J \subset [6], |J| = 4} H_2(\DD[\perp, D_J]),$$
    so we are interested in product 
    \begin{equation}\label{product1}
        H_1(\DD[\perp, D_I]) \otimes H_2(\DD[\perp, D_J]) \to 
        H_3(\DD[\perp, D_{I \cup J}]).
    \end{equation}
	For any $I \notin \K$, $|I| = 3$ the order complex $\Delta[\perp, D_I]$
	is the 1-simplex $\langle \perp, D_I \rangle$.
	So the generator of 
	$H_1(\DD[\perp, D_I])$ is the homology class
	$a = [\langle \perp, D_I \rangle]$.
	
	For $J$ such that $|J| = 4$ there are only two subsets $I_1, I_2 \subset J$, $|I_1| = |I_2| = 3$ that do not belong to $\K$.
	Hence, the 2-simplices of $\Delta[\perp, D_J]$ are
	$$\{\langle \perp, D_{I_1}, D_J \rangle, \langle \perp, D_{I_2}, D_J \rangle \},$$
	and $H_2(\DD[\perp, D_J])$ is generated by the homology class
	$b = [\langle \perp, D_{I_1}, D_J \rangle - \langle \perp, D_{I_2}, D_J \rangle ]$.

    The codimension condition in Theorem~\ref{product} implies that $ I \cup J = [6]$,
    so we have the following isomorphism
    \begin{equation}\label{Z-2}
    H_3(\DD[\perp, D_{I \cup J}]) \cong H_1(\Delta(\perp, D_{[6]})) 
        \cong H_1(\widehat{\K}') \cong \ZZ_2,
    \end{equation}
    where $\widehat{\K}'$ is the barycentric subdivision  of the Alexander dual complex $\widehat{\K}$.
    
    From the definition of the map $\times$ we have:
	$$ a \times b = 
	[ 
		\langle (\perp, \perp), (\perp, D_{I_1}),  (\perp, D_J), (D_I, D_J) \rangle - 
		\langle (\perp, \perp), (\perp, D_{I_1}), (D_I, D_{I_1}), (D_I, D_J) \rangle +
	$$
	$$
		\langle (\perp, \perp), (D_I, \perp), (D_I, D_{I_1}), (D_I, D_J) \rangle -
		\langle (\perp, \perp), (\perp, D_{I_2}),  (\perp, D_J), (D_I, D_J) \rangle +
	$$
	$$
		\langle (\perp, \perp), (\perp, D_{I_2}), (D_I, D_{I_2}), (D_I, D_J) \rangle -
		\langle (\perp, \perp), (D_I, \perp), (D_I, D_{I_2}), (D_I, D_J) \rangle
	].
	$$
	Since $| I \cap J | = 1$, we have $I_k \not\subset I$ for $k = 1,2$.
	Then $\vee_*$ maps the element $a \times b$ to %up to sign to
	$$[
		\langle \perp, D_{I_1}, D_J, D_{[6]} \rangle - \langle \perp, D_{I_1}, D_{I \cup I_1}, D_{[6]} \rangle +
		\langle \perp, D_I, D_{I \cup I_1}, D_{[6]} \rangle -
	$$
	$$
		\langle \perp, D_{I_2},  D_J, D_{[6]} \rangle +
		\langle \perp, D_{I_2}, D_{I \cup I_2}, D_{[6]} \rangle -
		\langle \perp, D_I, D_{I \cup I_2}, D_{[6]} \rangle ].
	$$
	This element is the generator of the group $H_3(\DD[\perp, D_{[6]}])$. For $I = \{1,2,3\}$, 
	$J = \{3,4,5,6\}$ its image under the isomorphism~(\ref{Z-2})
 %$H_3(\DD[\perp, D_{[6]}]) \cong H_1(\Delta(\perp, [6])) \cong %H_1(\widehat \K ')$, 
	%where $\widehat \K '$ is a barycentric subdivision of $\widehat \K$,
	is shown in the Figure~\ref{fig:RP2}.
    
    Thus, we conclude that product~(\ref{product1}) 
    %of the generators of groups $H^3(D(\K)) \cong \bigoplus n$ and $H^4(D(\K))$
	is non-trivial.
    %if and only if they correspond to the generators of $H^1(\K_I)$ and %$H^1(\K_J)$ for $I, J \notin \K$,
    %$|I| = 3$, $|J| = 4$, $|I \cap J| = 1$. Moreover, the product of %these elements is the generator of the group $H^7(D(\K))$.
\end{example}

\section{The connection between coordinate and diagonal arrangements}
First, we show that any coordinate arrangement complement can be realized as a diagonal arrangement complement up to homotopy.
Given a simplicial complex $\mathcal K$ on~$[m]$, define a new simplicial complex $\mathcal{L}$ on the vertex set $[m+1]$ as
	\begin{multline}\label{from-d-to-u}
		\mathcal{L} = \bigl\{ { \{i_1, \dots, i_k \}: \{i_1, \dots, i_k \} \subset [m] }\bigr\} \\
		\cup \bigl\{ { \{i_1, \dots, i_k, m+1\} \subset [m+1]: \{ i_1, \dots, i_k \} \in \K }\bigr\} .
	\end{multline}
 Geometrically, $\mathcal{L}$ is the union of an $(m-1)$-dimensional simplex on $[m]$ and the cone over $\K$ with apex~$\{m+1\}$.
Consider the corresponding diagonal arrangement complement $D(\mathcal{L}) \subset \CC^{m+1}$. 

\begin{theorem}\label{res1}
There is a homotopy equivalence
%	Let $\K$ be a simplicial complex on $[m]$.
%	Then there is a simplicial complex $\mathcal{L}$ on the vertex set $[m+1]$ such that
	$$U(\K) \simeq D(\mathcal{L}).$$
The same is true for real arrangement complements.
\end{theorem}
\begin{proof}
	The map
 %Define the following map:
	$$F: \CC^{m+1}\times [0,1] \to \CC^{m+1},$$
	$$F(z_1, \dots, z_{m+1}, t) = (z_1, \dots, z_{m+1}) - t
	\frac{z_1 + \dots + z_{m+1}}{m+1} (1, \dots, 1)$$
	defines a homotopy between the identity map $id: \CC^{m+1} \to \CC^{m+1}$ and the orthogonal projection
	$P: \CC^{m+1} \to \CC^{m+1}$ onto the hyperplane $\pi = \{z_1 + \dots + z_m = 0\}$.
	
Observe that $F(D(\mathcal{L}), t) \subset D(\mathcal{L})$ for every $t \in [0, 1]$. 
	Indeed, the equality $z_{i_1} = \dots = z_{i_k}$ holds if and only if
	$$z_{i_1} - t \frac{z_1 + \dots + z_{m+1}}{m+1} = \dots = z_{i_k} - t \frac{z_1 + \dots + z_{m+1}}{m+1}.$$
	Hence $z \notin D(\mathcal{L})$ implies $z \notin F(D(\mathcal{L}), t)$, that is $F(D(\mathcal{L}), t) \subset D(\mathcal{L})$.
	We also have that $F(z, t) = z$ for any $z \in \pi$, $t \in [0, 1]$,
	so we obtain $F(D(\mathcal{L}), 1) = D(\mathcal{L}) \cap \pi$,
	and therefore $D(\mathcal{L}) \simeq D(\mathcal{L}) \cap \pi$.
	
	Consider the linear isomorphism $A: \CC^{m+1} \to \CC^{m+1}$ defined by the formula
	$$y_1 = z_1 - z_{m+1},\quad \dots, \quad y_m = z_m - z_{m+1}, \quad y_{m+1} = z_1 + \dots + z_m + z_{m+1}.$$
 %$$\begin{array}{lcl}
	%	y_i = z_i - z_{m+1}, \text{ if } i = 1, \dots, m, \\
	%	y_{m+1} = z_1 + \dots z_m + z_{m+1}.
	%\end{array}$$
	We claim that $A(D(\mathcal{L}) \cap \pi) = U(\K)$, where $U(\K) \subset \{y_{m+1} = 0\} \subset \CC^{m+1}$. 
	Indeed, %the definition of $\K '$ implies that 
    the non-faces of $\mathcal{L}$ are of the form $\{i_1, \dots, i_k, m+1\}$,
	where $\{i_1, \dots, i_k\} \notin \K$. Therefore,
\begin{multline*}
  A(D(\mathcal{L}) \cap \pi) = A\Bigl( { \CC^{m+1} \backslash\!\!\!\!\bigcup_{\{i_1, \dots, i_k\} \notin \K}
		\{z_{i_1} = \dots = z_{i_k} = z_{m+1} \} }\Bigr)  \cap \,A \bigl({ \{z_1 + \dots + z_{m+1} = 0\} }\bigr) \\
		= \Bigl({\CC^{m+1} \backslash \bigcup_{\{i_1, \dots, i_k\} \notin \K}
		\{y_{i_1} = \dots = y_{i_k} = 0 \} }\Bigr) \cap \{y_{m+1} = 0 \} = U(\K),
\end{multline*}  
where the last equality holds since $U(\K)$ is considered as a subspace in the hyperplane $\{y_{m+1} = 0\}$.
It follows that $A\colon \CC^{m+1} \to \CC^{m+1}$ induces a homeomorphism between $D(\mathcal{L}) \cap \pi$ and $U(\K)$, so we obtain a homotopy equivalence $D(\mathcal{L}) \simeq D(\mathcal{L})\cap \pi \cong U(\K)$.
\end{proof}

Conversely, diagonal arrangement complements corresponding to a particular class of simplicial complexes can be realized as coordinate arrangement complements:

\begin{proposition}\label{res1.5}
	Suppose that $\mathcal{L}$ is a simplicial complex on the vertex set $[m]$ that has a face with $m-1$ vertices.
	Then there is a simplicial complex $\K$ such that $D(\mathcal{L}) \simeq U(\K)$.
\end{proposition}
\begin{proof}
	Assume that $\{1, 2, \dots, m-1 \} \in \mathcal{L}$.
	Consider the simplicial complex
	$\K = \link_{\mathcal{L}}\{m\} $ on the set $[m-1]$. 
	Since $\mathcal{L}$ can be obtained from $\K$ using operation~(\ref{from-d-to-u}) we have that $U(\K) \simeq D(\mathcal{L})$.
\end{proof}

The class of simplicial complexes $\mathcal L$ for which $D(\mathcal L)\simeq U(\mathcal K)$ can be extended by taking joins. Recall that the \textit{join} of
simplicial complexes $\K_1$ and $\K_2$ on the sets $V_1$ and $V_2$ is the following simplicial complex on the set $V_1\sqcup V_2$:
$$\K_1 * \K_2 = \{ I_1 \sqcup I_2: I_1 \in \K_1, I_2 \in \K_2 \}.$$
\begin{proposition}
%	Let $\K_1$ and $\K_2$ be simplicial complexes on the vertex sets $V_1$ and $V_2$, $|V_k| = m_k$. Then t
There is a homeomorphism
	$$D(\K_1 * \K_2) \cong D(\K_1) \times D(\K_2). $$
\end{proposition}
\begin{proof}
Note that $D(\K) = \CC^m \backslash \bigcup_{ I \notin \K } D_I= \CC^m \backslash \bigcup_{ I \in MF(\K)} D_I$, where the latter union is taken over the missing faces of~$\K$. Since $MF(\K_1 * \K_2) = MF(\K_1) \sqcup MF(\K_2)$, we get
\begin{multline*}D(\K_1 * \K_2) = \CC^m \backslash \Bigl({ \bigcup_{I \in MF(\K_1)} D_I \cup \bigcup_{J \in MF(\K_2)} D_J }\Bigr)\\ 
= \Bigl({ \CC^{m_1} \backslash \bigcup_{I \in MF(\K_1)} D_I }\Bigr) \times \Bigl({ \CC^{m_2} \backslash \bigcup_{J \in MF(\K_2)} D_J }\Bigr) = D(\K_1) \times D(\K_2). \qedhere
\end{multline*}
\end{proof}

\begin{example}
	Consider the simplicial complex that is the boundary of 4-gon:
	$\K = $
	\begin{picture}(37, 20)
		\multiput(10, 2)(15, 0){2}{\line(0, 1){15}}%
		\multiput(10, 2)(0, 15){2}{\line(1, 0){15}}%
		\put(10, 2){\circle*{3}} \put(25, 2){\circle*{3}} \put(10, 17){\circle*{3}} \put(25, 17){\circle*{3}}%
		\put(2, 0){4} \put(30, 0){3} \put(2, 13){1} \put(30, 13){2}
	\end{picture}.
	The corresponding diagonal arrangement complement has the form
	$$D(\K) = \CC^4 \backslash \left({ \{z_1 = z_3\} \cup \{z_2 = z_4\} }\right).$$
	Obviously
	$\K = \K_1 * \K_2$, where each of $\K_1$ and $\K_2$ is the pair of disjoint points.
	Each of these simplicial complexes satisfy the condition of Corollary \ref{res1.5}, so we can determine the homotopy type of $D(\K)$:
	$$D(\K) \cong D(\K_1) \times D(\K_2) \simeq U(\K_0) \times U(\K_0) \simeq S^1 \times S^1 = T^2,$$
	where $\K_0$ is the empty simplicial complex on the set $[2]$.
\end{example}

\begin{lemma}\label{sus}
	Let $X$ be a finite simplicial complex, and let $Y$ be a subcomplex in $X$. Then $\Sigma(X \backslash Y) \simeq (\Sigma X) \backslash Y$.
\end{lemma}
\begin{proof}
Consider the barycentric subdivision $X'$ of $X$ and the full subcomplex $Z$ in $X'$ on the vertices that do not belong to $Y$.
Define a neighbourhood $U_{X'}(Z)$ of $Z$ as the union of relative interiors of simplices of $X'$ having some simplex of $Z$ as a face.
Obviously, $Z$ is a deformation retract of $U_{X'}(Z)$. 

We claim that $X' \backslash U_{X'}(Z) = Y'$, where $Y'$ is the barycentric subdivision of $Y$.
Indeed, suppose that there is a point $x \in X'$ such that $x \notin Y'$ and consider the minimal face $F$ that contains $x$. 
Since we work with the barycentric subdivision of $X$, $F$ intersect $Y'$ by a single face.
Then there is a vertex $v \in F$ such that $v \notin Y'$, so $v \in Z$ and hence the relative interior of $F$ is a subset of $U_{X'}(Z)$.
The point $x$ belongs to the relative interior of $F$ since we took the minimal face $F$.

Thus we obtain a homotopy equivalence $X \backslash Y \simeq Z$, so $\Sigma(X \backslash Y) \simeq \Sigma Z$.
Applying the same arguments to the simplicial complex $\Sigma X$ and its subcomplex $Y$ we obtain that $(\Sigma X) \backslash Y \simeq \widetilde{Z}$,
where $\widetilde{Z}$ is the full subcomplex in the barycentric subdivision $(\Sigma X)'$ of $\Sigma X$ on the vertices that do not belong to $Y$. It remains to note that
$\widetilde{Z} = \Sigma Z$ as subcomplexes in $\Sigma X' = (\Sigma X)'$ to obtain the desired statement.
\end{proof}
\begin{theorem}\label{res2}
	Suppose that any two missing faces of simplicial complex $\K$ have a common vertex. Then there is a homotopy equivalence
	$$U(\K) \simeq \Sigma^2 D(\K).$$
	In  the case of real arrangement complements there is a homotopy equivalence 
 $$U_{\RR}(\K) \simeq \Sigma D_{\RR}(\K).$$
\end{theorem}
\begin{proof}
Let $D(\K)$ be the complement of the diagonal arranglement $\mathcal{DA}(\K)$.
	Denote by $\pi$ the hyperplane $\{z_1 + \dots + z_m = 0 \}$. The map
	$$F(z, t) = z - t \frac{z_1 + \dots + z_m}{m} (1, \dots, 1)$$
	defines a homotopy between the identity map on $D(\K)$ and the orthogonal projection of $D(\K)$ to $D(\K) \cap \pi$.
	The map
	$$G(z, t) = (1-t) z + t \frac{z}{|z|}$$
	defines a homotopy between the identity map on $D(\K) \cap \pi$ and
	the central projection from $D(\K) \cap \pi$ to $D(\K) \cap S^{2m-3}$, where $S^{2m-3}$ is the unit sphere in $\pi$.
Hence, we obtain a homotopy equivalence %$D(\K) \simeq D(\K) \cap S^{2m-3}$.
\begin{equation}\label{dk-to-s-pi}
D(\K) \simeq D(\K) \cap S^{2m-3} = S^{2m-3} \backslash (\mathcal{DA}(\K) \cap \pi)
\end{equation}
We define a continious function $f$ on the union $|\mathcal{DA}(\K)|$ by
\[ 
  f_{i_1, \dots, i_k}(z) = \frac{z_1 + \dots + z_m }{m} - z_{i_1}\quad\text{for }z\in\{z_{i_1} = \dots = z_{i_k} \}.
\]
These functions coincide on the intersections of the diagonal  subspaces in $\mathcal{DA}(\K)$. Indeed,
	for any two non-faces $I = \{i_1, \dots, i_k\}$ and $J = \{j_1, \dots, j_l \}$ of $\K$
	the intersection of subspaces $\{z_{i_1} = \dots = z_{i_k} \}$ 
	and $\{z_{j_1} = \dots = z_{j_l} \}$ has the form $\{z_{i_1} = \dots = z_{i_k} = z_{j_1} = \dots = z_{j_l} \}$,
	because $I$ and $J$ have a common vertex. So we obtain
	$$f_{i_1, \dots, i_k}(z) = \frac{z_1 + \dots + z_m }{m} - z_{i_1} = \frac{z_1 + \dots + z_m }{m} - z_{j_1} = f_{j_1, \dots, j_l}(z)$$
	for $z \in \{z_{i_1} = \dots = z_{i_k}\} \cap \{z_{j_1} = \dots = z_{j_l}\}$. 
Hence, $f$ is well defined on $|\mathcal{DA}(\K)|$.
	
The intersection $|\mathcal{DA}(\K)| \cap \pi$ is closed in $\pi$, so the Tietze theorem implies that the function $f$ can be continuously extended to a function $\widetilde{f}$ on the whole $\pi$ and then to a function $\widehat{f}$ on $\CC^m$
by the rule
	$$\widehat{f}(z) = \widetilde{f}(pr_{\pi}(z)),$$
where the $pr_{\pi}(z)$ is the orthogonal projection on the hyperplane $\pi$.
Then we define a map $\Phi \colon \CC^m \to \CC^m$ by the formula	$$\Phi (z) = z - \widehat{f}(z) (1, \dots, 1).$$
This map is a homeomorphism with the inverse map given by
	$$\Phi^{-1}(z) = z + \widehat{f}(z) (1, \dots, 1).$$
	
Consider the coordinate arrangement $\mathcal{CA}(K)$ and its complement $U(\K)$. We claim that 
\begin{equation}\label{uk-to-s-pi}
  \Phi(U(\K)) = \CC^m \backslash \left({ \mathcal{DA}(\K) \cap \pi }\right).
\end{equation}
Indeed, the restriction of $\widehat f$ to a coordinate subspace $\{z_{i_1} = \dots = z_{i_k} = 0 \}\in\mathcal{CA}(\K)$  is equal to $\frac{z_1 + \dots + z_m}{m}$. Hence,
\[
  \Phi(z) = z - \frac{z_1 + \dots + z_m}{m} (1, \dots, 1)
  \quad\text{for } z \in \mathcal{CA}(\K),
\]  
that is, $\Phi(z) \in \pi$. Also we have that $\Phi(\mathcal{DA}(\K)) \subset \mathcal{DA}(\K)$,	because $\Phi$ translates every point $z$ by a vector collinear to $(1, \dots, 1)$.
	Hence, $\Phi(\mathcal{CA}(\K) ) \subset \mathcal{DA}(\K) \cap \pi$.
	The inverse inclusion follows by applying the same arguments to the map $\Phi^{-1}$. 
	Since $\Phi$ is a homeomorphism and
	$\Phi(\mathcal{CA}(\K)) = \mathcal{DA}(\K) \cap \pi$, we obtain
	$\Phi(U(\K)) = \CC^m \backslash \left({ \mathcal{DA}(\K) \cap \pi }\right)$.
	
	The map $G(z, t)$ defines a homotopy between the identity map on $\Phi(U(\K))$ and the central projection
	from $U(\K)$ to $\Phi(U(\K)) \cap S^{2m-1}$, where $S^{2m-1}$ is the unit sphere in $\CC^m$. 
	Thus, we obtain homotopy equivalences:
	%$U(\K) \simeq S^{2m-1} \backslash \left({\mathcal{DA}(\K) \cap \pi }\right).$
\begin{multline*}
    \Sigma^2 D(\K) \simeq \Sigma^2(S^{2m-3} \backslash
    (\mathcal{DA}(\K) \cap \pi))
    \simeq S^{2m-1} \backslash(\mathcal{DA}(\K) \cap \pi)
    \simeq \CC^m \backslash (\mathcal{DA}(\K) \cap \pi) \cong U(\K).
\end{multline*}
Here the first homotopy equivalence follows from~\eqref{dk-to-s-pi}. The second equivalence follows Lemma~\ref{sus} (by the result of~\cite{BjZ}, the intersection %complement
of a subspace arrangement with a sphere can be viewed as a simplicial subcomplex). The last homeomorphism above is~\eqref{uk-to-s-pi}.
	
For real arrangement complements the argument is the same.
\end{proof}

Using Theorem~\ref{res2} we can reproduce some calculations of~\cite{BjW} for the cohomology groups of ``$k$-equal arrangement'' complements:
\begin{proposition}[\cite{BjW}, Theorem 1.1, 1.2]
	Suppose that $\K = \sk^{k-2} \Delta^m$, $k < m < 2k$. Then $D_\RR(\K)$ and $D(\K)$ are the real and complex ``$k$-equal arrangement'' complements, respectively.
 Their cohomology groups are given by
 \begin{align*}
		H^q(D_\RR(\K))& = 
		\begin{cases}
			\ZZ, &\text{if } q = 0,\\
			\ZZ^s, &\text{if } q = k-2, \text { where } s = \sum_{l = k}^{m} {{m} \choose l} {{l-1} \choose {k - 1}},\\
			0, &\text{otherwise},
		\end{cases}\\
H^q(D(\K)) &= 
		\begin{cases}
			\ZZ, &\text{if } q = 0,\\
			\ZZ^{t(q)}, &\text{if } 2k - 3 \le q \le m + k - 3, \text{ where }
   t(q) = {m \choose q-k+1} {q-k \choose k-1},\\
			0, &\text{otherwise},
		\end{cases}
\end{align*}
	%where $t(q) = {m \choose q-k+1} {q-k \choose k-1}$.
\end{proposition}
\begin{proof}
In~\cite{GT}, the following homotopy equivalences were proven:
	$$U_\RR(\K) \simeq \bigvee_{l = k}^m (S^{k-1})^{\vee {m \choose l} {l-1 \choose k-1}},\qquad 
 U(\K) \simeq \bigvee_{l = k}^m (S^{k + l - 1})^{\vee {m \choose l} {l-1 \choose k-1}}.$$
Now the statement follow from Theorem~\ref{res2} and the suspension isomorphism.
\end{proof}

\section{BBCG-decomposition and Golod complexes}
We refer to the stable decomposition of the polyhedral product described in the next theorem as the \emph{BBCG-decomposition}:
\begin{theorem}[\cite{BBCG}, Theorem 2.21]
	Let $\K$ be a simplicial complex on $[m]$, and the pairs of spaces $(\textbf X, \textbf A)$ have the property 
	that $X_i$ is contractible for all $i = 1, 2, \dots, m$.
	Then there is a homotopy equivalence.
	$$\Sigma (\textbf X, \textbf A)^\K \simeq \Sigma \bigvee_{\varnothing \neq I \subset [m]} |\K_I| * \widehat A^I.$$
\end{theorem}
In the particular case of $\Z_\K=(D^2,S^1)^\K$ we obtain
$$\Sigma \Z_\K \simeq \bigvee_{\varnothing \neq I \subset [m]} \Sigma^{|I| + 2}|\K_I|.$$

Necessary and sufficient conditions for desuspending the BBCG-decomposition were studied in \cite{IK}. One of the main results is the following:
%theorem.
\begin{theorem}[\cite{IK}, Theorem 1.3]\label{Kis}
	The following conditions are equivalent:
	\begin{enumerate}
		\item the fat wedge filtration of $\Z_\K$ is trivial;
		\item $\Z_\K$ is a co-$H$-space;
		\item there is a homotopy equivalence
			$$\Z_\K \simeq \bigvee_{\varnothing \neq I \subset [m]} \Sigma^{|I| + 1}|\K_I|.$$
	\end{enumerate}
\end{theorem}
We use this statement together with Theorem~\ref{res2} to obtain the decomposition of $\Sigma^2 D(\K)$.
\begin{proposition}\label{bbcg-dk}
	Suppose that any two missing faces of simplicial complex $\K$ on $[m]$ have a common vertex. Then the fat wedge filtration
	of $\Z_\K$ is trivial and there is a homotopy equivalence
	$$\Sigma^2 D(\K) \simeq \bigvee_{\varnothing \neq I \subset [m]} \Sigma^{|I| + 1} |\K_I|.$$
\end{proposition}
\begin{proof}
	Theorem~\ref{res2} implies that $\Z_\K \simeq \Sigma^2D(\K)$ is a co-H-space. Then Theorem~\ref{Kis} implies that the fat wedge filtration of $\Z_\K$ is trivial
	and the BBCG-decomposition of the moment-angle complex can be desuspended.
\end{proof}
In the general case it is impossible to desuspend the above decomposition and describe the homotopy type of $D(\K)$. The following example demonstrates such a case.
\begin{example}
Let $\K$ be the $6$-vertex triangulation of $\RR P^2$ from Example~\ref{D-RP2}. By the result of~\cite{GPTW}, the BBCG-decomposition
	can be desuspended in this case and we obtain:
	$$\Z_\K \simeq \bigvee_{\varnothing \neq I \subset [6]} \Sigma^{|I| + 1}|\K_I| 
	\simeq (S^5)^{\vee 10} \vee (S^6)^{\vee 15} \vee (S^7)^{\vee 6} \vee \Sigma^7(\RR P^2).$$
Since any two missing faces in $\K$ have a common vertex, Proposition~\ref{bbcg-dk} implies that $\Sigma^2 D(\K)\simeq U(\K)\simeq\Z_\K$ is a wedge of suspensions. However, this decomposition cannot be desuspended, because the multiplication in the ring $H^*(D(\K))$ is nontrivial as shown by Example~\ref{D-RP2}.
\end{example}

Recall that a simplicial complex $\K$ on the set $[m]$ is called \textit{Golod} over $\kk$, if all products and (higher) Massey products in $\Tor^*_{\kk [v_1, \dots, v_m]} (\kk [\K], \kk)$ vanish.

\begin{proposition}\label{Gol}
	Let $\K$ be a simplicial complex such that any two its missing faces have a common vertex. Then
	$\K$ is a Golod complex.
\end{proposition}
\begin{proof}
By Theorem~\ref{res2}, the moment-angle complex $\mathcal Z_\K$ is the double suspension of $D(\K)$.
Therefore, all products and higher Massey products in $H^*(\Z_\K) \cong \Tor^*_{\kk [v_1, \dots, v_m]}(\kk [\K], \kk)$ are trivial, so $\K$ is Golod.
\end{proof}
\begin{remark}
	This class of simplicial complexes includes the simplicial complexes described in Corollary~\ref{res1.5}.
\end{remark}

In \cite{IK}, two classes of simplicial complexes were constructed for which the fat wedge filtration
of the moment-angle complex is trivial. These are the \emph{homology fillable} and $\lceil \frac{\dim \K}{2} \rceil$-\emph{neighbourly} complexes. However, neither of this classes contains the class of complexes $\K$ in which any two missing faces have a common vertex. This is illustrated by the next example.

\begin{example}
Let  $\K$ be the simplicial complex on seven vertices with the missing faces given by
	\begin{eqnarray*}
		MF(\K) = \{ \{1, 2, 3, 7\}, \{1, 2, 4, 7\}, \{1, 3, 5, 7\}, \{1, 4, 6, 7\}, \{1, 5, 6, 7 \},\\
		\{2, 3, 6, 7 \}, \{2, 4, 5, 7\}, \{2, 5, 6, 7\}, \{3, 4, 5, 7\}, \{3, 4, 6, 7\} \}
	\end{eqnarray*}
	This simplicial complex can be obtained from the complex of Example~\ref{D-RP2} using the operation (\ref{from-d-to-u}).
    Then $\K$ is homotopy equivalent to the suspension of $\RR P^2$,
    %which
    so $\Sigma \K$ is not a wedge of spheres. 
    Hence, $\K$ is not homology fillable. We also have that $\dim \K = 5$, but $\K$ is not 3-neighbourly. Hence, $\K$ does not belong to either of the classes described in~\cite{IK}.
	Nevertheless, Proposition~\ref{bbcg-dk} and Proposition~\ref{Gol} imply that $\K$ is a Golod complex and that fat wedge filtration of $\Z_\K$ is trivial.
	
	As we can see, the BBCG-decomposition of moment-angle complex has the form
	\begin{equation}\label{bbcg-example}
		\Z_\K \simeq (S^7)^{\vee 10} \vee (S^8)^{\vee 15} \vee (S^9)^{\vee 6} \vee \Sigma^9(\RR P^2).
	\end{equation}
	Theorem~\ref{res1} implies that $D(\K) \simeq U(\mathcal{L})$, where $\mathcal{L}$ is the 6-vertex triangulation of projective plane.
	Thus, we can desuspend (\ref{bbcg-example}) and obtain the following decomposition for the diagonal arrangement complement:
	$$D(\K) \simeq (S^5)^{\vee 10} \vee (S^6)^{\vee 15} \vee (S^7)^{\vee 6} \vee \Sigma^7(\RR P^2).$$
\end{example}

\bibliographystyle{plain}
\bibliography{mybibliography}
\end{document}